\documentclass[11pt,a4paper]{amsart}
\usepackage{latexsym,amsmath,amssymb,amsthm,amscd}
\usepackage[alphabetic]{amsrefs}
\usepackage[all]{xy}
\theoremstyle{definition}
\newtheorem{Ex}{Example}[section]
\newtheorem{Def}[Ex]{Definition}
\newtheorem{Not}[Ex]{Notation}
\newtheorem{Rem}[Ex]{Remark}
\newtheorem{Case}{Case}
\theoremstyle{plain}
\newtheorem{Thm}[Ex]{Theorem}
\newtheorem{Prop}[Ex]{Proposition}
\newtheorem{Lem}[Ex]{Lemma}
\newtheorem{Cor}[Ex]{Corollary}
\newcommand{\Z}{\mathbb{Z}}
\newcommand{\N}{\mathbb{N}}

\DeclareMathOperator{\Hom}{Hom}
\DeclareMathOperator{\Mor}{Mor}
\DeclareMathOperator{\Aut}{Aut}
\newcommand{\Rep}{\mathrm{Rep}}
\newcommand{\GL}{\mathrm{GL}}

\DeclareMathOperator{\im}{Im}
\DeclareMathOperator{\id}{id}
\DeclareMathOperator{\Ker}{Ker}
\DeclareMathOperator{\res}{res}
\DeclareMathOperator{\Der}{Der}
\DeclareMathOperator{\ind}{ind}
\DeclareMathOperator{\soc}{soc}
\DeclareMathOperator{\I}{\mathcal{I}}
\DeclareMathOperator{\ch}{ch}
\DeclareMathOperator{\Char}{char}
\DeclareMathOperator{\Lie}{Lie}
\DeclareMathOperator{\Sym}{Sym}
\begin{document}
\newdir{ >}{{}*!/-5pt/@{>}}
\title[The Grothendieck ring of $G(n,r)$]{The Grothendieck ring of the
structure group of the geometric Frobenius morphism}

\author{Markus Severitt}

\begin{abstract}
The geometric Frobenius morphism on smooth varieties is an fppf-fiber bundle. 
We study representations of the structure group scheme. In particular, we describe irreducible 
representations and compute its Grothendieck ring of finite dimensional representations.  
\end{abstract}
\maketitle

\section{Introduction}
Let $k$ be a field of characteristic $p>0$. Then for all smooth $k$-varieties $X$ of dimension $n$, 
the $r$-th geometric Frobenius morphism
\[F^r:X\rightarrow X^{(r)}\]
is an fppf-fiber bundle with fibers $\mathbb{A}^n_r$, the $r$-th Frobenius kernel of the affine space 
of dimension $n$ considered as $\mathbb{G}_{a}^n$. Now denote
\[R(n,r):=k[\mathbb{A}^n_r]=k[x_1,\ldots,x_n]/(x_1^{p^r},\ldots,x_n^{p^r})\]
Let $G(n,r)$ be the automorphism group scheme of 
$R(n,r)$. Then for each $G(n,r)$-representation $V$, there is an associated canonical $X^{(r)}$-vector bundle by 
twisting the fiber bundle $F^r$ with $V$. In order to understand these bundles K-theoretically one needs to understand 
the Grothendieck ring of $G(n,r)$. The latter one is the purpose of this paper. The topic arose from a correspondence between 
Pierre Deligne and Markus Rost where Deligne suggested this setting for $n=r=1$.

We will give a description of the irreducible $G(n,r)$-representations whose classes form a $\Z$-basis of 
$K_0(G(n,r)\mathrm{-rep})$ as an abelian group. Note that $\Lie G(n,r)=\Der_k(R(n,r))$. That is, for $r=1$, 
this is the restricted Lie algebra of Cartan type Witt. In fact, the description and the involved computations we give 
generalize results of \cite{Nak92} for the simple restricted modules of $W_n=\Der_k(R(n,1))$. In order to get the description 
of irreducible representations, we need a triangular decomposition $G(n,r)=G^-G^0G^+$ with $G^0=\GL_n$.
In fact, most of the arguments involved are quite general. Hence we will give them in the abstract notion of 
\emph{triangulated groups}. Also, this notion turns out to be useful in order to describe the recursive part of our description which 
passes from $G(n,r)$ to $G(n,r+1)$. Furthermore, this notion also covers Jantzen's groups $G_rT$ and $G_rB$. Note that
the description of irreducible $G(n,1)$-representations was already given by Abrams \cite{Abr96}, \cite{Abr97}. 
Unfortunately, one part of our description only works if we exclude the case $\Char(k)=2$.

Our main goal is to describe $K_0(G(n,r)\mathrm{-rep})$ as a surjective image of $r+1$ copies of 
$K_0(\GL_n\mathrm{-rep})$ as an abelian group. This involves the de Rham complex of $R(n,r)$ over $k$ and Cartier's 
Theorem which computes its cohomology. We will also compute the kernel elements of our surjective map
\[K_0(\GL_n\mathrm{-rep})^{\oplus r+1}\rightarrow K_0(G(n,r)\mathrm{-rep})\]
and introduce a ring structure of the left hand side.

\section{Basic Properties}
Let us first fix some notions we are using. As already said, $k$ is a field of characteristic $p>0$. 
By an \emph{algebraic $k$-group $G$} we understand a group scheme $G$, 
that is a functor from commutative $k$-algebras to groups, which is representable by a finitely generated Hopf algebra. 
When we denote $g\in G$ we understand a choice of a commutative $k$-algebra $A$ and $g\in G(A)$. For the readabilty, we will 
suppress $A$. Furthermore, for $r\geq 0$, we denote by $G^{(r)}$ the \emph{$r$-th Frobenius twist} 
and by $G_r$ the kernel of the $r$-th Frobenius morphism $F^r:G\rightarrow G^{(r)}$.

Now the group scheme 
\[G(n,r)=\underline{\Aut}(R(n,r))\]
with
\[R(n,r):=k[\mathbb{A}^n_r]=k[x_1,\ldots,x_n]/(x_1^{p^r},\ldots,x_n^{p^r})\]
is defined to be
\[G(n,r)(A):=\Aut_{A\mathrm{-alg}}(R(n,r)_A)\]
with 
\[R(n,r)_A:=R(n,r)\otimes_k A=A[x_1,\ldots,x_n]/(x_1^{p^r},\ldots,x_n^{p^r})\]
Each element $g\in G(n,r)$ is determined by the images of the $x_i$. In fact, a choice $g_i\in R(n,r)_A$ for each $i=1,\ldots,n$ 
defines an element $g\in G(n,r)(A)$ by $g(x_i)=g_i$ if and only if $g_i(0)^{p^r}=0$ for all $i$ and $J_g\in \GL_n(A)$ where
\[J_g:=\left(\frac{\partial g_j}{\partial x_i}(0)\right)_{ij}\]
is the \emph{Jacobian matrix} of $g$.

As $\GL_n$ acts linearly on $\mathbb{A}^n_r$, we get 
\[G^0:=\GL_n\subset G(n,r)\] 
as a subgroup. Furthermore, 
$\mathbb{G}_{a,r}^n$, the $r$-th Frobenius kernel of $\mathbb{G}_a^n$, acts on $\mathbb{A}^n_r$ by translation. That is, we get 
\[G^-:=\mathbb{G}_{a,r}^n\subset G(n,r)\]
as a subgroup. Now denote
\begin{eqnarray*}
G^+&:=&\{g\in G(n,r)\mid\forall i:g(x_i)(0)=0,\ J_g=\id\}\\
&=&\{g\in G(n,r)\mid \forall i:g(x_i)=x_i+\sum_{I,\deg(I)\geq 2}\lambda_Ix^I\} 
\end{eqnarray*}
where $I\in\{0,\ldots,p^r-1\}^n$ is a multi-index with the usual degree $\deg(I)$ and $x^I\in R(n,r)$ is the corresponding monomial. 
Note that $G^+\cong\mathbb{A}^N$ for an $N\in\N$ as a scheme. These three subgroups provide a triangular decomposition
\[G(n,r)=G^-G^0G^+\]
That is, the multiplication $m:G^-\times G^0\times G^+\rightarrow G$ is an isomorphism of $k$-schemes.

Moreover, for all $1\leq i\leq n$, denote
\[U_i=U_i(n,r):=\{g\in G(n,r)\mid \forall j:g(x_j)(0)^{p^i}=0\}\subset G(n,r)\]
These subschemes are in fact subgroups who afford the triangular decomposition
\[U_i=G^-_iG^0G^+\]
Note that $G^-_i=\mathbb{G}_{a,i}^n$.

As we already noticed,
\[\Lie(G(n,r))=\Der_k(R(n,r))\]
the self-derivations of $R(n,r)$ by \cite{DG80}*{II\S4,2.3 Proposition}.
A canonical basis of this Lie algebra is given by the operators
\[\delta_{(i,x^I)}:=x^I\frac{\partial}{\partial x_i},\ I\in\{0,\ldots,p^r-1\}^n\]

\section{Triangulated Groups}
Now we will introduce the notion of triangulated groups and triangulated morphisms. 
It turned out to be a convenient notion in order to study the group schemes $G(n,r)$. 
Triangular decompositions are a standard tool in algebraic Lie theory. This notion is meant 
to catch some of their properties in an abstract way.
 
\begin{Def}
Let $H$ be an algebraic $k$-group. A \emph{pretriangulation} of $H$ is a collection of three subgroups $(H^-,H^0,H^+)$ 
such that the multiplication map
\[m:H^-\times H^0\times H^+\rightarrow H\]
is an isomorphism of $k$-schemes. We shortly denote this by
\[H=H^-H^0H^+\]
\end{Def}
As an example, we have 
\[G(n,r)=G^-G^0G^+\]
as well as $U_i(n,r)=G^-_iG^0G^+$.

Furthermore for each split reductive group $G$, we get for the $r$-th Frobenius kernel
\[G_r=U^-_rT_r U^+_r\]
where $T\subset G$ is a maximal torus and $U^{\pm}\subset G$ the unipotent subgroups.
\begin{Def}
Let $G=G^-G^0G^+$ and $H=H^-H^0H^+$ be pretriangulated and $f:G\rightarrow H$ a group homomorphism. 
Then $f$ is said to be \emph{triangulated}, if for all $\alpha\in\{-,0,+\}$ the restriction of 
$f$ to $G^{\alpha}$ factors through $H^{\alpha}$. 

We denote this factorization by $f^{\alpha}:G^{\alpha}\rightarrow H^{\alpha}$ and we write
\[f=f^-f^0f^+\]
\end{Def}
As an example consider a pretriangulation $H=H^-H^0H^+$. Then the $r$-th Frobenius twist of $H$ 
is pretriangulated by $H^{(r)}=(H^-)^{(r)}(H^0)^{(r)}(H^+)^{(r)}$ and the $r$-th Frobenius morphisms
\[F^r_H:H\rightarrow H^{(r)}\]
is triangulated with $(F^r_H)^{\alpha}=F^r_{H^{\alpha}}$.

The following Lemma is straightforward.
\begin{Lem}\label{prop of tr mor}
Let $f:G\rightarrow H$ be triangulated. Then the following holds:
\begin{enumerate}
 \item The kernel of $f$ is pretriangulated by 
\[\Ker(f)=\Ker(f^-)\Ker(f^0)\Ker(f^+)\]
\item The image of $f$ is pretriangulated by 
\[\im(f)=\im(f^-)\im(f^0)\im(f^+)\]
\item The closed immersion $\im(f)\hookrightarrow H$ is triangulated. 
\end{enumerate}
\end{Lem}
Note that under the isomorphism $G/\Ker(F)\cong \im(f)$, we also a obtain a pretriangulation
\[G/\Ker(f)=G^-/\Ker(f^-)G^0/\Ker(f^0)G^+/\Ker(f^+)\]

As an example we take a pretriangulation $H=H^-H^0H^+$ and the $r$-th Frobenius morphism $F^r_H:H\rightarrow H^{(r)}$. 
Then we get that the $r$-th Frobenius kernel of $H$ is pretriangulated by
\[H_r=H_r^-H_r^0H_r^+\]

Our next aim is to describe irreducible representation of a pretriangulated group $H$ in terms of $H^0$. 
In order to do that, we need the following.
\begin{Def}
A pretriangulation $H=H^-H^0H^+$ is called a \emph{triangulation} if the following statements hold:
\begin{enumerate}
 \item The following products are semi direct by conjugation
\[B^-:=H^-\rtimes H^0\ \mbox{and}\ B^+:=H^0\ltimes H^+\]
\item $H^-$ and $H^+$ are unipotent.
\item $H^-$ is finite.
\end{enumerate}
\end{Def}
Note that this definition is not symmetric. 
As an example, the pretriangulation $G(n,r)=G^-G^0G^+$ is also a triangulation.

For a triangulation $H=H^-H^0H^+$, we denote the group homomorphisms
\begin{enumerate}
 \item the projections $\pi^{\pm}:B^{\pm}\rightarrow H^0$
 \item the inclusions $j^{\pm}:B^{\pm}\hookrightarrow H$
\end{enumerate}
and the functor
\[\I:=(j^+)_{\ast}(\pi^+)^{\ast}:H^0\mathrm{-rep}\rightarrow H\mathrm{-rep}\]
between categories of finite dimensional representations. Here 
\[(\pi^+)^{\ast}:H^0\mathrm{-rep}\rightarrow B^+\mathrm{-rep}\]
and 
\[(j^+)_{\ast}=\Mor_{B^+}(H,-):B^+\mathrm{-rep}\rightarrow H\mathrm{-rep}\]
is induction. That is
\[(j^+)_{\ast}(V)=\{f\in\Mor(H,V)\mid f(hg)=hf(g)\ \forall g\in H,\ \forall h\in B^+\}\]
and the $H$-action is induced by right translation on $H$. The functor $(j^+)_{\ast}$ preserves finiteness as
\[\Mor_{B^+}(H,-)\cong\Mor(H^-,-)\]
by using the decomposition $H=B^+ H^-$ and the finiteness of $H^-$. In order to express the $H$-action, let us 
denote for $h\in H$ the decomposition
\[h=h_+h_0h_-\]
according to the decomposition $H=B^+H^-=H^+H^0H^-$. The following Lemma is straightforward.
\begin{Lem}\label{desc of I-functor}
Let $H=H^-H^0H^+$ be a triangulation and $V$ an $H^0$-repre\-sen\-ta\-tion.
\begin{enumerate}
 \item Under the isomorphism
\[\I(V)\cong\Mor(H^-,V)\]
the $H$-action translates as follows: For all $h\in H$, $a\in H^-$, and $f:H^-\rightarrow V$
\[(hf)(a)=(ah)_0f((ah)_-)\]
\item If we restrict to $B^-$, we get 
\[(j^-)^{\ast}\I(V)\cong k[H^-]\otimes_k V\]
Here $H^-$ acts on $k[H^-]$ by the right regular 
representation and trivial on $V$ and $H^0$ acts on $k[H^-]$ by conjugation and as given on $V$. 
\end{enumerate}
\end{Lem}
In particular, we get
\[\I(V)^{H^-}\cong (k[H^-]\otimes_k V)^{H^-}=k[H^-]^{H^-}\otimes_k V\cong V\]
as $H^0$-representations.
\begin{Ex}\label{Gnr-action on I}
Let us consider the group scheme $G(n,r)$ and the triangulation $G(n,r)=G^-G^0G^+$. 
Recall that $G^-=(\mathbb{G}_{a,r})^n$ and $G^0=\GL_n$. Then for all $g\in G(n,r)$, we get 
$g_-=g(0)\in (\mathbb{G}_{a,r})^n$ and $g_0=J_g\in \GL_n$.
As $R(n,r)=k[(\mathbb{G}_{a,r})^n]$, we get for each 
$\GL_n$-representation $V$ that
\[\I(V)\cong R(n,r)\otimes_k V\]
as $k$-vector spaces.
The $G(n,r)$-action reads as follows:
\[g(f\otimes v)=\left( \frac{\partial g(x_j)}{\partial x_i}\right )_{ij}(g(f)\otimes v)\]
for all $g\in G(n,r)$, $f\in R(n,r)$, and $v\in V$. Note that this uses the fact that for $g\in G(n,r)(A)$ we get 
$\left( \frac{\partial g(x_j)}{\partial x_i}\right )_{ij}\in \GL_n(R(n,r)_A)$ which acts on $R(n,r)_A\otimes_k V$. 

In particular, take $V=k$ with the trivial $\GL_n$-action. Then we get $\I(k)\cong R(n,r)$ together with the standard action of 
$G(n,r)$ on $R(n,r)$. That is, the standard action can be recovered from the triangulation.
\end{Ex}
Now the functor $\I$ gives rise to the following description of irreducible $H$-representations.
\begin{Prop}\label{prop: param of irred H rep for triang}
Let $H=H^-H^0H^+$ be a triangulation. Then the maps
\[\soc\I:\{\mathrm{irred.}\ H^0\mathrm{-rep}\}/_{\cong}\longrightarrow \{\mathrm{irred.}\ H\mathrm{-rep}\}/_{\cong}\]
and
\[(-)^{H^-}:\{\mathrm{irred.}\ H\mathrm{-rep}\}/_{\cong}\longrightarrow\{\mathrm{irred.}\ H^0\mathrm{-rep}\}/_{\cong}\]
are well-defined and inverse to each other.
\end{Prop}
\begin{proof}
Let $V$ be an irreducible $H^0$-representation. As $H^-$ is unipotent, a standard argument 
by taking $H^-$-invariants and using $\I(V)^{H^-}\cong V$ shows that 
$\soc\I(V)$ is an irreducible $H$-representation. Also, $\soc\I(V)^{H^-}\cong V$ follows. Now let $W$ be an 
irreducible $H$-representation. As $H^+$ is unipotent $(W^{\vee})^{H^+}\neq 0$. Thus there is an irreducible 
$H^0$-representation $V\neq 0$ such that $V^{\vee}\subset (W^{\vee})^{H^+}$. Hence by dualization
\[0\neq \Hom_{B^+}((j^+)^{\ast}W,(\pi^{+})^{\ast}V)\cong\Hom_H(W,\I V)\]
This shows $W\cong\soc\I(V)$ and finishes the proof. 
\end{proof}
Finally, we extend the notion of triangulations as follows.
\begin{Def}
A triangulation $H=H^-H^0H^+$ is called an \emph{$r$-triangulation} if $H^-$ is of height $\leq r$. That is, $H^-$ equals its $r$-th Frobenius kernel:
\[H^-=(H^-)_r\]
\end{Def}
Note that an $r$-triangulation is also an $r+1$-triangulation. As an example, the triangulation $G(n,r)=G^-G^0G^+$ 
is also an $r$-triangulation as $G^-=\mathbb{G}_{a,r}^n$. Moreover $U_i(n,r)=G^-_iG^0G^+$ is an $i$-triangulation.

Now for an $r$-triangulation $H=H^-H^0H^+$, the $r$-th 
Frobenius factors as
\[F^r_H:H\rightarrow (B^+)^{(r)}\subset H^{(r)}\]
which provides the group homomorphism
\[P_r:=\pi^+\circ F^r_H:H\rightarrow (H^0)^{(r)}\]
In the example $G(n,r)$, this computes as
\[P_r(g)=F^r_{\GL_n}(J_g)=\left(\left( \frac{\partial g(x_j)}{\partial x_i}(0)\right)^{p^r}\right)_{ij}\]
If we consider $(H^0)^{(r)}$ triangulated with trivial $\pm$-factors, $P_r$ is triangulated with $P_r^{\pm}$ 
trivial and $P_r^0=F^r_{H^{0}}$. Thus
\[\Ker(P_r)=H^-H^0_rH^+\]
Moreover, $P_r$ has the following very useful property.
\begin{Lem}
Let $H=H^-H^0H^+$ be an $r$-triangulation such that $H^0$ is reduced. Then $P_r$ induces an isomorphism
\[H/\Ker(P_r)\cong (H^0)^{(r)}\]
\end{Lem}
\begin{proof}
This follows from the triangulated structure of $P_r$, Lemma \ref{prop of tr mor} and the fact that $F^r_{H^0}$ 
induces an isomorphism
\[H^0/(H^0)_r\cong (H^0)^{(r)}\]
if $H^0$ is reduced \cite{Jan03}*{I.9.5}.
\end{proof}
As an immediate consequence, we get that for $H^0$ reduced, the functor
\[P_r^{\ast}:(H^0)^{(r)}\mathrm{-rep}\longrightarrow H\mathrm{-rep}\]
preserves irreducible representations.
\begin{Not}
For any algebraic $k$-group $G$ and a $G^{(r)}$-representation $W$, we denote the $r$-th Frobenius twist by
\[W^{[r]}:=(F^r_G)^{\ast}W\]
\end{Not}
Then we get the following computational rule.
\begin{Lem}\label{Tensor Id for I}
Let $H=H^-H^0H^+$ be an $r$-triangulation, $V$ an $H^0$-re\-pre\-sen\-ta\-tion, and $W$ an $(H^0)^{(r)}$-representation. Then
\[\I(V\otimes_k W^{[r]})\cong \I(V)\otimes_k P_r^{\ast}W\]
as $H$-representations.
\end{Lem}
\begin{proof}
The claim follows from the Tensor Identity \cite{Jan03}*{I.3.6} for induction and the fact that 
$P_r^{\ast}W|_{H^0}=W^{[r]}$.
\end{proof}
This provides the following Proposition which reads as and uses Steinberg's Tensor Product Theorem 
\cite{Jan03}*{II.3.16,II.3.17}. For this, we need the following notations where for split reductive groups we follow \cite{Jan03}. 
\begin{Not}
Let $H=H^-H^0H^+$ be a triangulation such that $H^0$ is split reductive. Denote by $T\subset H^0$ 
a split maximal torus and by $X(T)_+$ the \emph{dominant weights}. Furthermore denote by $S$ the \emph{simple roots} 
and for $r\geq 1$
\[X_r(T):=\{\lambda\in X(T)\mid \forall \alpha\in S: 0\leq \langle \lambda,\alpha^{\vee}\rangle<p^r\}\] 
For $\lambda\in X(T)_+$, we denote the associated irreducible $H^0$-representation by $L(\lambda)$. 
Then, according to Proposition \ref{prop: param of irred H rep for triang}, we denote the associated irreducible 
$H$-representation by
\[L(\lambda,H):=\soc\I(L(\lambda))\] 
\end{Not}
Note that for an $r$-triangulation $H=H^-H^0H^+$ with $H^0$ split reductive, we get for all $\lambda\in X(T)_+$
\[L(p^r\lambda,H)\cong P_r^{\ast}L(\lambda)\]
as $P_r^{\ast}$ preserves irreducible representations and
\[(P_r^{\ast}L(\lambda))^{H^-}=P_r^{\ast}L(\lambda)|_{H^0}=L(\lambda)^{[r]}\cong L(p^r\lambda)\] 
\begin{Prop}\label{analogue of Steinbergs TPT}
Let $H=H^-H^0H^+$ be an $r$-triangulation such that $H^0$ is split reductive, $\lambda\in X_r(T)$ and $\mu\in X(T)_+$. 
Then
\[L(\lambda+p^r\mu,H)\cong L(\lambda,H)\otimes_k P_r^{\ast}L(\mu)\cong L(\lambda,H)\otimes_k L(p^r\mu,H)\]
\end{Prop}
\begin{proof}
By Steinberg's Tensor Product Theorem
\[L(\lambda+p^r\mu)\cong L(\lambda)\otimes_k L(\mu)^{[r]}\]
Then the previous Lemma provides
\[\I (L(\lambda+p^r\mu))\cong \I (L(\lambda))\otimes_k P_r^{\ast}L(\mu)\]
As $P_r^{\ast}L(\mu)=(P_r^{\ast}L(\mu))^{H^-}$, the result follows by taking socles.
\end{proof}
Note that examples of such $r$-triangulations with reductive $H^0$ are given by Jantzen`s groups $G_rT$ and $G_rB$. 
The Proposition just generalizes results which are already known for these groups.

\section{Transfer Homomorphisms}
\label{section:transfer homs}
In order to prepare the description of the irreducible $G(n,r)$-re\-presenta\-tions, we will introduce several 
\emph{transfer homomorphisms}. They will be  between the $G(n,r)$, $U_i(n,r)$, and $G(n,r)^0\cong\GL_n$ and their Frobenius twists respectively. 

In the previous section, we already introduced the morphism
\[P_r:G(n,r)\rightarrow (\GL_n)^{(r)}\]
arising from the $r$-triangulation $G(n,r)=G^-G^0G^+$. We also saw that for all $\lambda\in X(T)_+$ we have
\[P_r^{\ast}L(\lambda)\cong L(p^r\lambda,G(n,r))\] 

Now, for $r\geq2$ we also introduce a transfer morphism
\[T_r:G(n,r)\rightarrow G(n,r-1)^{(1)}\]
as follows: Let $S\subset R(n,r)$ be the subalgebra generated by $x_1^{p^{r-1}},\ldots,x_n^{p^{r-1}}$. Now let 
$g\in G(n,r)=\underline{\Aut}(R(n,r))$. Then $S$ is invariant under $g$. Thus we get an induced
automorphism on
\[R(n,r)\otimes_{S,(-)^p}k \cong R(n,r-1)\otimes_{k,(-)^p}k=R(n,r-1)^{(1)}\]
which defines $T_r(g)$. Note that if $F_a:R(n,r)\rightarrow R(n,r)$ denotes the 
\emph{arithmetic Frobenius}, then
\[T_r(g)(x_i\otimes 1)=g(x_i)\otimes 1=(F_a(g(x_i)))(x_i\otimes 1)\]
Now $T_r$ is triangulated as $(T_r)^-$ is just the first Frobenius morphism 
\[F^1:\mathbb{G}_{a,r}^n\rightarrow (\mathbb{G}_{a,r-1}^n)^{(1)}\]
and $(T_r)^0$ is also the first Frobenius morphism
\[F^1:\GL_n\rightarrow (\GL_n)^{(1)}\]
Note that $(T_r)^+$ is not just the first Frobenius as we set $x_i^{p^{r-1}}=0$ for all $i=1,\ldots,n$.
\begin{Lem}\label{Tr surjective}
For all $r\geq 2$, the morphism $T_r:G(n,r)\rightarrow G(n,r-1)^{(1)}$ induces an isomorphism
\[G(n,r)/\Ker(T_r)\xrightarrow{\cong} G(n,r-1)^{(1)}\]
\end{Lem}
\begin{proof}
By the triangulated structure of $T_r$ and Lemma \ref{prop of tr mor}, its suffices to show the claim for 
$(T_r)^-,(T_r)^0,(T_r)^+$ separately. For $(T_r)^-$ it follows by \cite{Jan03}*{I.9.5} which states that 
it induces an isomorphism
\[\mathbb{G}_{a,r}^n/\mathbb{G}_{a,1}^n\cong (\mathbb{G}_{a,r-1}^n)^{(1)}\]
Also by \cite{Jan03}*{I.9.5}, it follows for $(T_r)^0$: It induces an isomorphism
\[\GL_n/(\GL_n)_1\cong (\GL_n)^{(1)}\]
It is left to show that the closed immersion
\[T_r^+:G(n,r)^+/\Ker(T_r^+)\hookrightarrow (G(n,r-1)^+)^{(1)}\]
is an isomorphism. The describing ideal of this immersion is the kernel of the morphism
\[(T_r^+)^{\#}:k[G(n,r-1)^+]^{(1)}\rightarrow k[G(n,r)^+]\]
As the parameters for $G(n,r)^+$ are free and $T_r$ acts as the $p$-th power on the parameters, 
$(T_r^+)^{\#}$ is injective which shows the claim.
\end{proof}
As an immediate consequence, we get that
\[T_r^{\ast}:G(n,r-1)^{(1)}\mathrm{-rep}\longrightarrow G(n,r)\mathrm{-rep}\]
preserves irreducible representations. More concretely
\[T_r^{\ast}L(\lambda,G(n,r-1)^{(1)})\cong L(p\lambda,G(n,r))\]
for all $\lambda\in X(T)_+$. This follows by Proposition \ref{prop: param of irred H rep for triang} 
and the observation that
\begin{eqnarray*}
L(p\lambda)=L(\lambda)^{[1]}&=&(F_{\GL_n}^1)^{\ast}\left(L(\lambda,G(n,r-1)^{(1)})^{(G(n,r-1)^{(1)})^-}\right)\\
&\subset&(T_r^{\ast}L(\lambda,G(n,r-1)^{(1)}))^{G(n,r)^-} 
\end{eqnarray*}

Furthermore we introduce the transfer homomorphisms
\[t_{r,i}:U_i(n,r)\rightarrow G(n,i)\]
for all $1\leq i\leq r$. Recall that
\[U_i=\{f\in G(n,r)\mid f(0)\in\mathbb{G}_{a,i}^n\}\]
So let $g\in U_i$. Then it induces an isomorphism on $R(n,i)$ which we denote by $t_{r,i}(g)$. Note that 
\[t_{r,i}(g)(x_i)\equiv g(x_i)\mod (x_1^{p^i},\ldots,x_n^{p^i})\]
Also $t_{r,i}$ is triangulated: The restriction to $G^-_i$ and $G^0$ is just the identity.

Finally, we discuss how the maps $P_r$, $T_r$, $t_{r,i}$ are related. First note that for all $r\geq 2$, the diagram
\[\xymatrix{G(n,r)\ar[dr]_-{P_r}\ar[r]^-{T_r}&G(n,r-1)^{(1)}\ar[d]^{P_{r-1}}\\
 &(G^0)^{(r)}}\]
commutes. Furthermore, for all $1\leq i\leq r$, the diagram
\[\xymatrix{U_i(n,r)\ar[dr]_{P_i}\ar[r]^-{t_{r,i}}&G(n,i)\ar[d]^{P_i}\\
 &(G^0)^{(i)}}\]
commutes.

We again denote $G^-=G(n,r)^-$ and $U_i=U_i(n,r)\subset G(n,r)$. Recall that $U_i^-=G^-_i$. 
Our next aim is to study the induction functor $\ind_{U_i}^{G(n,r)}$ and its relation to the induced functors 
of the three morphism types. 
\begin{Lem}\label{exactness of ind Ui}
For all $1\leq i\leq r$, we get for the induction functor
\[\res_{G^-}^{G(n,r)}\circ\ind_{U_i}^{G(n,r)}=\ind_{G^-_i}^{G^-}\circ \res_{G^-_i}^{U_i}\]
Furthermore $\ind_{U_i}^{G(n,r)}$ is exact. 
\end{Lem}
\begin{proof}
We use the morphism description of the induction $\ind_{U_i}^{G(n,r)}$. Then we obtain for all 
$U_i$-representations $V$ that
\begin{eqnarray*}
\ind_{U_i}^{G(n,r)}V&=&\{f\in\Mor(G(n,r),V)\mid f(ug)=uf(g)\ \forall u\in U_i\}\\
&=&\{ f\in\Mor(G^-,V)\mid f(cg)=cf(g)\ \forall c\in G^-_i\}
\end{eqnarray*}
by using the decomposition $G(n,r)=B^+\times G^-$ and $U_i=B^+\times G^-_i$ where $B^+=G(n,r)^0\ltimes G(n,r)^+$. 
The restriction of this to $G^-$ coincides with 
\[\ind_{G^-_i}^{G^-}\res_{G^-_i}^{U_i}V\]
as the $G^-$-action is given by right translation. This shows the first claim. 
Now $\ind_{G^-_i}^{G^-}$ is exact by \cite{Jan03}*{I.5.13, I.9.5}. This implies the exactness 
of $\ind_{U_i}^{G(n,r)}$.
\end{proof}
We start with the relation of $\ind_{U_i}^{G(n,r)}$ to the $\I$-functors. Recall that $G(n,r)$ is $r$-triangulated and $U_i$ is $i$-triangulated. So let us denote the $\I$-functor for a $j$-triangulated group as $\I_j$.
\begin{Lem}\label{comm of ind Ui with I}
For all $1\leq i\leq r$, both triangles of the diagram
\[\xymatrix{&G(n,i)\mathrm{-rep}\ar[rd]^{t_{r,i}^{\ast}}&\\
 G^0\mathrm{-rep}\ar[rr]^{\I_i}\ar[ur]^{\I_i}\ar[dr]_{\I_r}&&U_i(n,r)\mathrm{-rep}\ar[dl]^{\ind_{U_i}^{G(n,r)}}\\
&G(n,r)\mathrm{-rep}&}
\]
commute.
\end{Lem}
\begin{proof}
The commutativity of the upper triangle follows immediately from $G(n,i)^-=U_i(n,r)^-$ and Lemma \ref{desc of I-functor}.

The commutativity of the lower triangle follows from $G(n,r)=B^+\times G^-$, $U_i=B^+\times G^-_i$, and the transitivity of induction \cite{Jan03}*{I.3.5}.
\end{proof}

Finally there is a more complicated relation of the induction $\ind_{U_i}^{G(n,r)}$ to the functors $P_i^{\ast}$, $T_j^{\ast}$, and $\I$:
\begin{Lem}\label{comm of ind Ui with I and L}
For all $1\leq i\leq r$, the triangle and the square of the following diagram commute up to functor isomorphism
\[\xymatrix{&G(n,i)\mathrm{-rep}\ar[rd]^-{t_{r,i}^{\ast}}&\\
(G^0)^{(i)}\mathrm{-rep}\ar[ru]^-{P_i^{\ast}}\ar[rr]^{P_i^{\ast}}\ar[d]_{\I_{r-i}}&&U_i(n,r)\mathrm{-rep}\ar[d]^{\ind_{U_i}^{G(n,r)}}\\
G(n,r-i)^{(i)}\mathrm{-rep}\ar[rr]^-{(T^i)^{\ast}}&&G(n,r)\mathrm{-rep}}\]
Here $T^i:G(n,r)\rightarrow G(n,r-i)^{(i)}$ is the composition
\[T^i:=T_{r-(i-1)}^{(i-1)}\circ\cdots\circ T_r\]
\end{Lem}
\begin{proof}
The commutativity of the triangle follows from $P_i\circ t_{r,i}=P_i$. 

For the commutativity of the square, note that the morphism
\[T^i=T_{r-(i-1)}^{(i-1)}\circ\cdots\circ T_r\]
is triangulated with 
\[(T^i)^-=F^i_{G^-}:G^-\rightarrow (G^-)^{(i)}_{r-i}\]
and 
\[(T^i)^0=F^i_{G^0}:G^0\rightarrow (G^0)^{(i)}\] 

Recall the morphism description of the functor $\I$ from Lemma \ref{desc of I-functor}. Let $V$ be a $(G^0)^{(i)}$-representation. On one hand
\[(T^i)^{\ast}\I_{r-i}(V)=\Mor((G^-)^{(i)}_{r-i},V)\]
as $(G(n,r-i)^{(i)})^-=(G^-)^{(i)}_{r-i}$. On the other hand, $G^-_i$ operates trivially on $P_i^{\ast}V$ which implies
\begin{eqnarray*}
\ind_{U_i}^{G(n,r)}P_i^{\ast} V&=&\{ f\in\Mor(G^-,P_i^{\ast}V)\mid f(cg)=cf(g)\ \forall c\in G^-_i\}\\
&=&\Mor(G^-/G^-_i,P_i^{\ast}V) 
\end{eqnarray*}
(cf.~the proof of Lemma \ref{exactness of ind Ui}). As the $i$-th Frobenius $F^i_{G^-}:G^-\rightarrow(G^-)^{(i)}_{r-i}$ induces an isomorphism $G^-/G^-_i\cong (G^-)^{(i)}_{r-i}$, 
it induces a natural linear isomorphism 
\[(T^i)^{\ast}\I_{r-i}(V)=\Mor((G^-)^{(i)}_{r-i},V)\xrightarrow{(F^i)^{\ast}}\Mor_{G^-_i}(G^-,P_i^{\ast}V)=\ind_{U_i}^{G(n,r)}P_i^{\ast} V\]
This isomorphism is in fact $G(n,r)$-equivariant which can be seen by using the triangulated structure of $T^i$. 
\end{proof}

\section{Differentials and Cartier's Theorem}
We are now going to introduce some concrete $G(n,r)$-representations which play a major role in the description of 
the irreducible representations. We consider the K\"ahler-differentials
\[\Omega_r:=\Omega_{R(n,r),k}=\bigoplus_{i=1}^nR(n,r)dx_i\]
We claim that this is a canonical $G(n,r)$-representation.
\begin{Not}
For any $g\in G(n,r)$ and any $R(n,r)$-module $M$, we denote by $M^{(g)}$ the
module \emph{twisted by $g$}, that is
\[x\ast^{(g)} m:=g(x)m\]
\end{Not}
We obtain that
\[R(n,r)\xrightarrow{g}R(n,r)^{(g)}\xrightarrow{d}\Omega_r^{(g)}\]
is an $R(n,r)$-derivation which induces an $R(n,r)$-module automorphism
\[\partial g:\Omega_r\rightarrow \Omega_r^{(g)}\]
This reads as
\[\partial g(fdx_i)=g(f)dg(x_i)\]
and provides a canonical $G(n,r)$-action on $\Omega_r$ as a $k$-vector space.
Moreover this operation extends to exterior and symmetric powers over $R(n,r)$
as
well as tensor products. In particular, we obtain a representation by the
$i$-th higher differentials
\[\Omega^i_r:=\Lambda^i_{R(n,r)}\Omega_r=\bigoplus_{j_1<\ldots<j_i}R(n,r)dx_{j_1}\wedge\ldots\wedge dx_{j_i}\]
They are connected by the \emph{de Rham complex}
\[0\rightarrow R(n,r)\xrightarrow{d_1}\Omega^1_r\xrightarrow{d_2}\cdots\xrightarrow{d_n}\Omega_r^n\rightarrow 0\]
The differential maps are defined by
\[d_i(fdx_{j_1}\wedge\ldots\wedge dx_{j_i}):=df\wedge dx_{j_1}\wedge\ldots\wedge dx_{j_i}\]
In fact the maps $d_i$ are $G(n,r)$-equivariant which can be shown by induction on $i$.
\begin{Rem}
Note that with $U=k^n$, we canonically get
\[\Omega_r^i\cong R(n,r)\otimes_k \Lambda^i U\cong \I_r(\Lambda^i U)\]
according to Example \ref{Gnr-action on I}. Hence, for all $(\GL_n)^{(r)}$-representations $V$, we obtain
\[\I_r(\Lambda^i U\otimes V^{[r]})\cong \Omega_r^i\otimes P_r^{\ast}V\]
by Lemma \ref{Tensor Id for I}. That is, we also have a twisted de Rham complex $\Omega_r^{\bullet}\otimes P_r^{\ast}V$. 
\end{Rem}
Now we use the transfer morphism $t_{r,j}:U_j(n,r)\rightarrow G(n,j)$ in order to compare the de Rham complexes 
$\Omega_j^{\bullet}$ and $\Omega_r^{\bullet}$ by the functor
\[\ind_{U_j}^{G(n,r)}\circ t_{r,j}^{\ast}:G(n,j)\mathrm{-rep}\longrightarrow G(n,r)\mathrm{-rep}\]
\begin{Prop}\label{compare deRhams}
For all $1\leq j<r$, we get
\[\ind_{U_j}^{G(n,r)}(t_{r,j}^{\ast}\Omega_j^{\bullet})\cong \Omega_r^{\bullet}\]
as complexes. Furthermore 
\[\ind_{U_j}^{G(n,r)}(t_{r,j}^{\ast}H^i(\Omega_j^{\bullet}))\cong H^i(\Omega_r^{\bullet})\]
for all $0\leq i\leq n$.
\end{Prop}
\begin{proof}
By Lemma \ref{comm of ind Ui with I} and the previous remark, we get canonical isomorphisms
\begin{eqnarray*}
\ind_{U_j}^{G(n,r)}(t_{r,j}^{\ast}\Omega_j^{i})&\cong&\ind_{U_j}^{G(n,r)}(t_{r,j}^{\ast}\I_j\Lambda^i U)\\
&\cong&\I_r(\Lambda^i U)\\
&\cong&\Omega^i_r
\end{eqnarray*}
Similarly, one can check by a tedious exercise that also
\[\ind_{U_j}^{G(n,r)}(t_{r,j}^{\ast}(d_i:\Omega_j^{i-1}\rightarrow \Omega_j^i))\cong(d_i:\Omega_r^{i-1}\rightarrow\Omega_r^{i})\]
The claim about the cohomology follows from the exactness of $\ind_{U_j}^{G(n,r)}$ Lemma \ref{exactness of ind Ui}.
\end{proof}
That is, we can compute the cohomology of the complex $\Omega_r^{\bullet}$ by the cohomology of the complex 
$\Omega_1^{\bullet}$. 

Before we do this, we need an additional observation. Let $f^r:R(n,r)\rightarrow R(n,r)$ the $r$-th power of the absolute Frobenius. 
It factors as
\[R(n,r)\xrightarrow{f^r}k\hookrightarrow R(n,r)\]
as $P^{p^r}=P(0)^{p^r}$ for all $P\in R(n,r)$. This provides induced $G(n,r)$-representations
\[\Omega_r^i\otimes_{R(n,r),f^r}k\cong \Lambda^i U\otimes_{k,f^r}k=\Lambda^iU^{(r)}\]
where again $U=k^n$. 
\begin{Lem}
For all $1\leq i\leq n$, we get
\[\Omega_r^i\otimes_{R(n,r),f^r}k\cong P_r^{\ast}\Lambda^i U^{(r)}\]
\end{Lem}
\begin{proof}
Observe that the group homomorphism
\[G(n,r)\rightarrow \GL(U^{(r)})=(\GL_n)^{(r)}\]
corresponding to the $G(n,r)$-representation $\Omega_r\otimes_{R(n,r),f^r}k$ coincides with $P_r$. This provides the claim for $i=1$. 
The claim for $i\geq 2$ follows from this by compatibility with exterior powers.
\end{proof}
We get a representation-theoretic reformulation of Cartier's famous theorem about the cohomology of the de Rham complex. It follows from 
its proof in \cite{Kat70}*{Theorem 7.2} and the previous lemma.
\begin{Thm}[Cartier]
There is a unique collection of isomorphisms of $G(n,1)$-representations 
\[C^{-1}:P_1^{\ast}\Lambda^i U^{(1)}\rightarrow H^i(\Omega_1^{\bullet})\]
which satisfies
\begin{enumerate}
 \item $C^{-1}(1)=1$
 \item $C^{-1}(\omega\wedge\tau)=C^{-1}(\omega)\wedge C^{-1}(\tau)$
 \item $C^{-1}(df\otimes 1)=[f^{p-1}df]\in H^1(\Omega_1^{\bullet})$
\end{enumerate}
\end{Thm}
\begin{Rem}
In fact, the proof in \cite{Kat70} does not provide the property that the $C^{-1}$ are $G(n,1)$-equivariant. 
But by property $(2)$, it suffices to check it for $i=1$ which follows from property $(3)$. 
\end{Rem}
As announced before, we can deduce a computation of the cohomology of $\Omega_r^{\bullet}$ for $r\geq 2$ with help
of the transfer morphism $T_r:G(n,r)\rightarrow G(n,r-1)^{(1)}$.
\begin{Cor}
For all $r\geq 2$ and $1\leq i\leq n$, we get an isomorphism
\[H^i(\Omega_r^{\bullet})\cong T_r^{\ast}((\Omega_{r-1}^i)^{(1)})\]
of $G(n,r)$-representations.
\end{Cor}
\begin{proof}
According to Proposition \ref{compare deRhams}, Cartier's Theorem, and Lemma \ref{comm of ind Ui with I and L}, we obtain
\begin{eqnarray*}
H^i(\Omega_r^{\bullet})&\cong&\ind_{U_1}^{G(n,r)}(t_{r,1}^{\ast}(H^i(\Omega_1^{\bullet})))\\
&\cong&\ind_{U_1}^{G(n,r)}(t_{r,1}^{\ast}(P_1^{\ast}\Lambda^iU^{(1)}))\\
&\cong&T_r^{\ast}(\I_{r-1}(\Lambda^i U^{(1)}))\\
&\cong&T_r^{\ast}((\Omega_{r-1}^i)^{(1)})
\end{eqnarray*}
Whence the claim.
\end{proof}
Finally we want to twist the de Rham complex $\Omega_r^{\bullet}$ with an $(\GL_n)^{(1)}$-representation $V$. 
For $r=1$, we already introduced the twist
\[\Omega_1^{\bullet}\otimes_l P_1^{\ast} V\cong \I_1(\Lambda^{\bullet}U\otimes V^{[1]})\]
By Cartier's Theorem, its cohomology computes as
\[H^i(\Omega_1^{\bullet})\otimes_k P_1^{\ast}V\cong P_1^{\ast}(\Lambda^iU^{(1)}\otimes_k V)\] 
Again, we consider the functor
\[\ind_{U_1}^{G(n,r)}\circ t_{r,1}^{\ast}:G(n,1)\mathrm{-rep}\longrightarrow G(n,r)\mathrm{-rep}\]
According to Lemma \ref{comm of ind Ui with I}, it provides a complex
\[\ind_{U_1}^{G(n,r)}(t_{r,1}^{\ast}\I_1(\Lambda^{\bullet}U\otimes V^{[1]}))=\I_r(\Lambda^{\bullet}U\otimes V^{[1]})\]
of $G(n,r)$-representations. As $\GL_n$-representations, this complex reads as
\[\Omega_r^{i-1}\otimes_k V^{[1]}\xrightarrow{d_i\otimes\id}\Omega_r^i\otimes_k V^{[1]}\]
Similarly as in the previous Corollary, we obtain
\[H^i(\Omega_r^{\bullet}\otimes_k V^{[1]})\cong T_r^{\ast}((\Omega_{r-1}^i)^{(1)}\otimes_k V)\]
for its cohomology.

\section{Irreducible Representations}\label{section:irred Gnr reps}
We are now going to compute the irreducible $G(n,r)$-representations
\[L(\lambda,G(n,r))=\soc\I(L(\lambda))=G(n,r)L(\lambda)\subset \I(L(\lambda))\]
for all $\lambda\in X(T)_+$ with respect to their associated irreducible $\GL_n$-re\-pre\-sen\-ta\-tions. 
For this, we take as split maximal torus the diagonal matrices. This torus affords canonical projections $\varepsilon_i\in X(T)$ for $1\leq i\leq n$ 
which are a free $\Z$-basis of the character group $X(T)$.

According to Proposition \ref{analogue of Steinbergs TPT} there is a $\mathrm{mod}\ p^r$-periodicity for the dominant
weights and one can restrict to the $L(\lambda,G(n,r))$ with $\lambda\in
X_r(T)$. This will cover the case $r=1$. The case $r\geq 2$ is more subtle.

We restrict to the following subset $X_1'(T)\subset X_1(T)$:
\[X_1'(T):=\left\{\lambda=\sum_{i=1}^n m_i(\epsilon_1+\ldots+\epsilon_i)\in
X(T)\mid \forall\ 1\leq i\leq n: 0\leq m_i< p\right\}\]
\begin{Not}
As $X_1'(T)$ is a set of representatives for $X(T)/pX(T)$, we get a unique
decomposition
\[\lambda=r(\lambda)+ps(\lambda)\]
for all $\lambda\in X(T)_+$ with $r(\lambda)\in X_1'(T)$ and $s(\lambda)\in X(T)_+$. We call $r(\lambda)$
the \emph{$\mathrm{mod}\ p$-reduction} of $\lambda$.
\end{Not}

The following Proposition covers the dominant weights $\lambda$ with $r(\lambda)=0$.
\begin{Prop}\label{soc for rlambda=0}
Let $\lambda\in X(T)_+$. Then we obtain
\[L(p\lambda,G(n,1))\cong P_1^{\ast}L(\lambda)\]
and for $r\geq 2$ we get
\[L(p\lambda,G(n,r))\cong T_r^{\ast}L(\lambda,G(n,r-1)^{(1)})\]
\end{Prop}
\begin{proof}
The claim for $r=1$ is just a special case of Proposition \ref{analogue of Steinbergs TPT} and the claim 
for $r\geq 2$ follows by the discussion after Lemma \ref{Tr surjective}.
\end{proof}

The idea for the case $r(\lambda)\neq 0$ is the following: The $G(n,r)^-$-invariants of the socle 
of $\I(L(\lambda))=L(\lambda)\otimes R(n,r)$ are $L(\lambda)\otimes k\cong L(\lambda)$. That is, the socle is generated 
by this subspace as a $G(n,r)$-representation. 
\begin{Not}
For $\lambda\in X_1'(T)$ write $\lambda=\sum_{i=1}^n m_i(\epsilon_1+\ldots+\epsilon_i)$ where $0\leq m_i<p$
and consider the $\GL_n$-representation
\[\Sym^{m_1}(U)\otimes_k \Sym^{m_2}(\Lambda^2 U)\otimes_k \ldots\otimes_k \Sym^{m_n}(\Lambda^n U)\]
where $U=k^n$ with canonical basis $e_1,\ldots,e_n$. Now consider the vector
\[v(\lambda)=e_1^{m_1}\otimes(e_1\wedge e_2)^{m_2}\otimes\ldots\otimes (e_1\wedge\ldots\wedge e_n)^{m_n}\]
in this representation. We define $W(\lambda)$ to be the $\GL_n$-subrepresentation generated by this vector.

For a general $\lambda\in X(T)_+$ set
\[W(\lambda):=W(r(\lambda))\otimes L(s(\lambda))^{[1]}\] 
\end{Not}
Note that $W(r(\lambda))$ has highest weight $r(\lambda)$ of multiplicity $1$. That is, there is a subrepresentation 
$V\subset W(r(\lambda))$ such that
\[L(r(\lambda)) \cong W(r(\lambda))/V\]
Thus
\[L(\lambda)\cong W(\lambda)/(V\otimes_k L(s(\lambda))^{[1]})\]
by Steinberg's Tensor Product Theorem. For example if $\lambda=\varepsilon_1+\ldots+\varepsilon_i$ is a fundamental weight, then
\[L(\varepsilon_1+\ldots+\varepsilon_i)=\Lambda^i U=W(\lambda)\]
As $\I$ is exact, we get 
\[\I(L(\lambda))=\I(W(\lambda))/\I(V\otimes_k L(s(\lambda))^{[1]})\]
Since $v(r(\lambda))$ is a highest weight vector, it follows that 
\begin{eqnarray*}
\soc\I(L(\lambda))
&=&G(n,r)L(\lambda)\\
&=&G(n,r)(v(r(\lambda))\otimes_k L(s(\lambda))^{[1]})/\I(V\otimes_k L(s(\lambda))^{[1]}) 
\end{eqnarray*}
Finally, note that
\begin{eqnarray*}
 &&\I(\Sym^{m_1}(U)\otimes_k\Sym^{m_2}(\Lambda^2 U)\otimes_k\ldots\otimes_k
\Sym^{m_n}(\Lambda^n U))\\
&=&R(n,r)\otimes_k\Sym^{m_1}(U)\otimes_k\Sym^{m_2}(\Lambda^2
U)\otimes_k\ldots\otimes_k \Sym^{m_n}(\Lambda^n U)\\
&\cong&\Sym_{R(n,r)}^{m_1}(\Omega_r^1)\otimes_{R(n,r)}\Sym_{R(n,r)}^{m_2}
(\Omega_r^2)\otimes_{R(n,r)}\ldots\otimes_{R(n,r)}
\Sym_{R(n,r)}^{m_n}(\Omega_r^n)
\end{eqnarray*}
as $G(n,r)$-representations and for $\lambda=r(\lambda)$, the vector $v(\lambda)$ corresponds to
\[v=(dx_1)^{m_1}\otimes(dx_1\wedge dx_2)^{m_2}\otimes\ldots\otimes (dx_1\wedge\ldots dx_n)^{m_n}\]

In order to compute $G(n,r)v\subset \I(W(\lambda))$, we will use the Lie algebra operators $\delta_{(i,x^I)}$. 
The following Proposition describes their action on $\I(V)$.
\begin{Lem}\label{Lie action on I(V)}
Let $\delta_{(i,x^I)}\in\Lie G(n,r)$ be a canonical basis element. Then for all $\GL_n$-representations $V$, 
the induced action on
\[\I(V)=R(n,r)\otimes_k V\]
reads as 
\[\delta_{(i,x^I)}(f\otimes v)=
\left(x^I\frac{\partial}{\partial x_i}f\right)\otimes v +\sum_{j=1}^nf\frac{\partial}{\partial x_j}x^I\otimes E_{ji}(v)\]
where $E_{ji}\in \mathrm{M}_n(k)=\Lie\GL_n$ is the $(j,i)$-th standard matrix.
\end{Lem}
\begin{proof}
First recall the $G(n,r)$-action on $\I(V)=R(n,r)\otimes_k V$:
\[g(f\otimes v)=\left( \frac{\partial g(x_s)}{\partial x_k}\right )_{ks}(g(f)\otimes v)\]
for all $g\in G(n,r)$, $f\in R(n,r)$, and $v\in V$.

In order to compute the action of $\delta_{(i,x^I)}=x^I\frac{\partial}{\partial x_i}$, we consider the corresponding element 
$g_{(i,I)}=1+\delta_{(i,x^I)}\epsilon\in G(n,r)(k[\epsilon])$ 
where $k[\epsilon]$ are the dual numbers. That is, 
\[g_{(i,I)}(x_s)=\begin{cases}
                  x_s + x^I\epsilon&s=i\\
                  x_s&s\neq i
                 \end{cases}\]
Now we get
\[\delta_{(i,x^I)}(f\otimes v)=\frac{\partial}{\partial\epsilon}\left(\left( \frac{\partial g_{(i,I)}(x_s)}{\partial x_k}\right )_{ks}(g_{(i,I)}(f)\otimes v)\right)\Big |_{\epsilon=0}\]
The product rule provides
\[\delta_{(i,x^I)}(f\otimes v)=\delta_{(i,x^I)}(f)\otimes v + 
f\frac{\partial}{\partial\epsilon}\left(\left( \frac{\partial g_{(i,I)}(x_s)}{\partial x_k}\right )_{ks}(1\otimes v)\right)\Big |_{\epsilon=0}\]
As 
\[\frac{\partial g_{(i,I)}(x_s)}{\partial x_k }=\begin{cases}
                  1 + \frac{\partial x^I}{\partial x_k}\epsilon&s=i=k\\
                  \frac{\partial x^I}{\partial x_k}\epsilon&s=i\neq k\\
                  1&s=k\neq i\\
                  0&s\neq k,s\neq i
                 \end{cases}\]
we get
\[\frac{\partial}{\partial\epsilon}\left(\frac{\partial}{\partial x_k }g_{(i,I)}(x_s) \right)\Big |_{\epsilon=0}
=\begin{cases}
\frac{\partial x^I}{\partial x_k}&s=i\\
0&s\neq i                                                                                                   
\end{cases}
\]
Whence the claim.
\end{proof}
\begin{Rem}
Note that the action of $\delta_{(i,x^I)}$ on $W(\lambda)=W(r(\lambda))\otimes L(s(\lambda))^{[1]}$ is 
$(-)\otimes_k\id_{L(s(\lambda)^{[1]}}$ applied to the action on $W(r(\lambda))$ as $\Lie\GL_n$ acts trivially 
on Frobenius twists $V^{[1]}$.
\end{Rem}
Now we are ready to treat the case where the $\mathrm{mod}\ p$-reduction of $\lambda$ is a 
fundamental weight $\varepsilon_1+\ldots+\varepsilon_i$.
By Steinberg's Tensor Product Theorem, we know that
\[L(\lambda)\cong L(\epsilon_1+\ldots+\epsilon_i)\otimes_k
L(s(\lambda))^{[1]}\cong\Lambda^i U\otimes_k L(s(\lambda))^{[1]}=W(\lambda)\]
with $U=k^n$. That is, 
\[\I(L(\lambda))\cong \Omega_r^i\otimes_k L(s(\lambda))^{[1]}\]
which is part of the twisted de Rham complex as introduced at the end of the previous section. Recall that the 
differentials read as
\[\Omega_r^{i-1}\otimes_k L(s(\lambda))^{[1]}\xrightarrow{d_i\otimes\id}\Omega_r^i\otimes_k L(s(\lambda))^{[1]}\]
where $d_i:\Omega_r^{i-1}\rightarrow\Omega_r^i$ is the de Rham-differential.
\begin{Prop}\label{soc for rlambda=fw}
Let $\lambda\in X(T)_+$ with $r(\lambda)=\epsilon_1+\ldots+\epsilon_i$, then 
\[L(\lambda,G(n,r))\cong \soc(\Omega_r^i\otimes_k
L(s(\lambda))^{[1]})=\im(d_i\otimes \id)\]
where $d_i:\Omega_r^{i-1}\rightarrow\Omega_r^i$ is the de Rham-differential.
\end{Prop}
\begin{proof}
We will use $\Lie(G(n,r))$-operators to prove the claim. Recall that $f\in\Lie G(n,r)$ acts on 
$\Omega_r^i\otimes_kL(s(\lambda))^{[1]}$ as $f\otimes\id$.

We already noticed that the socle is generated
by the $G(n,r)^-$-invariants as a $G(n,r)$-representation. A generating system of
these invariants is given by
\[(dx_{j_1}\wedge\ldots\wedge dx_{j_i})\otimes v\]
for all $j_1<\ldots<j_i$ and $v\in L(s(\lambda))^{[1]}$. Now the inclusion 
\[\soc(\Omega_r^i\otimes L(s(\lambda))^{[1]})\subset \im(d_i\otimes\id)\] 
follows from the fact that the generators lie in the image of $d_i\otimes\id$. For the inclusion $\im(d_i\otimes\id)\subset \soc(\Omega_r^i\otimes
L(s(\lambda))^{[1]})$ note that $\im(d_i\otimes\id)$ is as a $k$-vector space
generated by the elements
\[(dx^I\wedge dx_{j_1}\wedge\ldots\wedge dx_{j_{i-1}})\otimes v\]
where $v\in L(s(\lambda))^{[1]})$, $x^I=x_1^{m_1}\cdots x_n^{m_n}\in R(n,r)$, and
$j_1<\ldots<j_{i-1}$. As $i-1<n$, there is an index
$l\notin\{j_1,\ldots,j_{i-1}\}$. Then we get that the Lie algebra operator
$\delta_{(l,x^I)}\in \Lie(G(n,r))$ acts as
\[\delta_{(l,x^I)}((dx_{l}\wedge dx_{j_1}\wedge\ldots\wedge dx_{j_{i-1}})\otimes
v)=(dx^I\wedge dx_{j_1}\wedge\ldots\wedge dx_{j_{i-1}})\otimes v\]
according to Lemma \ref{Lie action on I(V)}. This provides all image elements
from the generators. 
\end{proof}
The next Proposition covers the case where the $\mathrm{mod}\ p$-reduction of $\lambda$ is neither $0$ nor a 
fundamental weight if we assume
$\Char(k)\neq 2$.
\begin{Prop}\label{soc for rlambda not 0 or fw}
Assume that $\Char(k)\neq 2$. Let $\lambda\in X(T)_+$ a dominant weight with
$r(\lambda)\neq 0$ and $r(\lambda)\neq\epsilon_1+\ldots+\epsilon_i$ for all
$i=1,\ldots,n$. Then
\[L(\lambda,G(n,r))=\I(L(\lambda))\]
\end{Prop}
Before we give the proof, we need some technical Lemmas.
\begin{Lem}\label{L1}
Let $V$ be $\GL_n$-representation, $v\in V$, and $1\leq s_j\leq p^{r-1}$. If 
\[x_1^{ps_1-1}\cdots x_n^{ps_n-1}v\in G(n,r)v\subset R(n,r)\otimes_k V=\I(V)\]
then 
\[x^Jv\in G(n,r)v\subset \I(V)\]
for all $J=(j_1,\ldots,j_n)$ with $p(s_k-1)\leq j_k<ps_k$ for all $1\leq
k\leq n$.
\end{Lem}
\begin{proof}
This follows by the gradual application of the Lie algebra operators
\[\delta_{i}=\frac{\partial}{\partial x_i}\otimes\id:R(n,r)\otimes
V\rightarrow R(n,r)\otimes V\qedhere\]
\end{proof}
\begin{Lem}\label{L2}
Let $V$ be a $\GL_n$-representation, $v\in V$, and $1\leq s_j\leq p^{r-1}$. Assume that $x^Jv\in G(n,r)v\subset R(n,r)\otimes_k V=\I(V)$ 
for $J=(j_1,\ldots,j_n)$.
\begin{enumerate}
 \item If $j_k=ps$, then for all $j\neq k$ we get
\[x^JE_{jk}v\in G(n,r)v\subset\I(V)\]
\item If $j_i=ps$ and $x_j\frac{\partial}{\partial x_k}x^Jv\in G(n,r)v$ we get
\[x^JE_{ki}E_{jk}v\in G(n,r)v\subset\I(V)\]
\end{enumerate}
\end{Lem}
\begin{proof}
The first part follows from
\[x^JE_{jk}v=\delta_{(k,x_j)}(x^Jv)\]
as $j_k=ps$.

The second part follows from
\begin{eqnarray*}
x^JE_{ki}E_{jk}v&=&\delta_{(i,x_k)}(x^JE_{jk}v)\\
&=&\delta_{(i,x_k)}\left(\delta_{(k,x_j)}(x^Jv)-x_j\frac{\partial}{\partial x_k}x^Jv\right) 
\end{eqnarray*}
as $j_i=ps$.
\end{proof}
\begin{Lem}\label{L3}
Let $\lambda\in X_1'(T)$ and write $\lambda=\sum_{i=1}^n m_i(\epsilon_1+\ldots+\epsilon_i)$. 
Let $v=v(\lambda)\in W(\lambda)$ the vector from above. Let $k$ be the highest index, such that $m_k\neq 0$ 
and $i<k$ the highest index such that $m_i\neq 0$.
\begin{enumerate}
 \item For all $j\leq k$, we get
\[E_{jk}v=\delta_{jk}m_kv\]
where $\delta_{jk}$ is the Kronecker-$\delta$.
 \item For $j>k$, we get
\[E_{kk}E_{jk}v=(m_k-1)E_{jk}v\]
\item If $m_k=1$, we get
\[E_{ii}v=(m_i+1)v\]
\item If $m_k=1$, we get for all $j>k$
\[E_{ii}E_{ki}v=m_iE_{ki}v\]
and
\[E_{ii}E_{ki}E_{jk}v=m_iE_{ki}E_{jk}v\]
\end{enumerate}
\end{Lem}
\begin{proof}
Recall that
\[v=e_1^{m_1}\otimes(e_1\wedge e_2)^{m_2}\otimes\ldots\otimes (e_1\wedge\ldots e_k)^{m_k}\]
The first claim follows by
\[E_{jk}v=m_ke_1^{m_1}\otimes(e_1\wedge e_2)^{m_2}\otimes\ldots\otimes (e_1\wedge\ldots\wedge e_{k-1}\wedge e_{j})(e_1\wedge\ldots\wedge e_k)^{m_k-1}\]

The second follows from the first by
\[E_{kk}E_{jk}v=E_{jk}E_{kk}v+(E_{kk}E_{jk}-E_{jk}E_{kk})v,\]
using $E_{kk}E_{jk}=0$ and $E_{jk}E_{kk}=E_{jk}$. 

Now let $m_k=1$. Then the third claim follows as $E_{ii}$ acts precisely on the factors $(e_1\wedge\ldots\wedge e_i)^{m_i}$ and $(e_1\wedge\ldots\wedge e_k)$ of $v$.

The claim $E_{ii}E_{ki}v=m_iE_{ki}v$ follows from the third in the same fashion as the second follows from the first.

Finally for the last claim using $E_{ii}E_{jk}=E_{jk}E_{ii}=0$ and the third claim, we get
\[E_{ii}E_{jk}v=E_{jk}E_{ii}v=(m_i+1)E_{jk}v\]
Hence
\begin{eqnarray*}
E_{ii}E_{ki}E_{jk}v&=&E_{ki}E_{ii}E_{jk}v-E_{ki}E_{jk}v\\
&=&m_iE_{ki}E_{jk}v 
\end{eqnarray*}
which finishes the proof.
\end{proof}

\begin{proof}[Proof of \ref{soc for rlambda not 0 or fw}]
By previous discussions, it suffices to prove 
\[G(n,r)(v(r(\lambda))\otimes_k L(s(\lambda))^{[1]})=\I(W(r(\lambda))\otimes_k L(s(\lambda))^{[1]})=\I(W(\lambda))\]
Again we will use $\Lie(G(n,r))$-operators to prove the claim. As again $f\in\Lie G(n,r)$ acts as $f\otimes\id$, 
we can assume that $\lambda=r(\lambda)$.

Let $v=v(\lambda)$ as above. It suffices to show
\[R(n,r)\otimes_k kv(\lambda)\subset G(n,r)v(\lambda)\subset\I(W(\lambda))\]
Now write again
\[\lambda=\sum_{i=1}^n m_i(\epsilon_1+\ldots+\epsilon_i)\]
As $\lambda=r(\lambda)$, we have $0\leq m_i\leq p-1$ for all $i=1,\ldots,n$.

By Lemma \ref{L1}, it suffices to prove
\[x_1^{ps_1-1}\cdots x_n^{ps_n-1}v\in G(n,r)v\]
for all choices $1\leq s_j\leq p^{r-1}$.

By assumption we have $\lambda\neq 0$. That is, there is a highest index $k$ such
that $m_k\neq 0$. 

\begin{Case}[$m_k\geq 2$] Let us assume that $m_k\geq 2$. We argue by descending induction on $s_k$. 

Take $I=(ps_1-1,\ldots,ps_n-1)$. Note that for all $J=(j_1,\ldots,j_n)$ with $j_k\geq ps_k$ we get
\begin{eqnarray}
x^Jv&\in& G(n,r)v\label{IHc1}
\end{eqnarray}
For $s_k=p^{r-1}$ this is clear as in this case $x^J=0$. The case $s_k<p^{r-1}$ follows by the induction hypothesis 
and Lemma \ref{L1}.

By Lemma \ref{Lie action on I(V)} we get
\begin{eqnarray}
\delta_{(k,x^I)}(v)&=&m_k\frac{\partial}{\partial x_k}x^Iv+\sum_{j>k}\frac{\partial}{\partial x_j} x^IE_{jk}v\label{Fc1}
 \end{eqnarray}
since $E_{jk}v=\delta_{jk}m_kv$ for $j\leq k$ by Lemma \ref{L3}(1). 

Now we apply the operator
$\delta_{(k,x_k^2)}$ to $\delta_{(k,x^I)}(v)$. Let $j>k$. Using $E_{kk}E_{jk}v=(m_k-1)E_{jk}v$ Lemma \ref{L3}(2) we obtain
\begin{eqnarray*}
\delta_{(k,x_k^2)}\left(\frac{\partial}{\partial x_j} x^I
E_{jk}v\right)&=&x_k^2\frac{\partial}{\partial x_k}\frac{\partial}{\partial
x_j}x^IE_{jk}v+2x_k\frac{\partial}{\partial x_j} x^IE_{kk}E_{jk}v\\
&=&x_k^2\frac{\partial}{\partial x_k}\frac{\partial}{\partial
x_j}x^IE_{jk}v+2(m_k-1)x_k\frac{\partial}{\partial x_j} x^IE_{jk}v\\
&\in&G(n,r)v
\end{eqnarray*}
by (\ref{IHc1}) and Lemma \ref{L2}(1).

As $\delta_{(k,x_k^2)}(\delta_{(k,x^I)}(v))\in G(n,r)v$, we obtain
\[\delta_{(k,x_k^2)}\left(\frac{\partial}{\partial x_k}x^Iv\right)\in G(n,r)v\]
by (\ref{Fc1}). Using $E_{kk}v=m_kv$, we get
\begin{eqnarray*}
\delta_{(k,x_k^2)}\left(\frac{\partial}{\partial x_k} x^Iv\right)&=&x_k^2\frac{\partial}{\partial x_k}\frac{\partial}{\partial x_k}
x^Iv+2x_k\frac{\partial}{\partial x_k}
x^IE_{kk}v\\
&=&\left(x_k^2\frac{\partial}{\partial x_k}\frac{\partial}{\partial x_k}
x^I+2m_kx_k\frac{\partial}{\partial x_k}
x^I\right)v\\
&=&(ps_k-1)(ps_k-2+2m_k)x^Iv\\
&=&2(1-m_k)x^I v\\&\in& G(n,r)v
\end{eqnarray*}
As $2\leq m_k\leq p-1$ and $\Char(k)=p\neq 2$, we have $2(1-m_k)\neq 0$. Hence
\[x_1^{ps_1-1}\cdots x_n^{ps_k-1} v=x^I v\in G(n,r)v\]
for all choices $1\leq s_j\leq p^{r-1}$ which finishes the proof for the case $m_k\geq
2$.
\end{Case}
\begin{Case}[$m_k=1$]
In the second case, we assume that $m_k=1$. As $r(\lambda)=\lambda$ is not a
fundamental weight by assumption, there is a highest index $i<k$ with $m_i\neq
0$. Here, we will argue by descending induction on $s_i+s_k$.

Take again $I=(ps_1-1,\ldots,ps_n-1)$. Note that for all $J=(j_1,\ldots,j_n)$ with $j_k\geq ps_k\wedge j_i\geq p(s_i-1)$ 
or $j_k\geq p(s_k-1)\wedge j_i\geq ps_i$, we get
\begin{eqnarray}
x^Jv&\in& G(n,r)v\label{IHc2}
\end{eqnarray}
For $s_k=p^{r-1}$ or $s_i=p^{r-1}$, this is clear as in this case $x^J=0$. In the case $j_k\geq ps_k$ and $j_i\geq p(s_i-1)$ 
and $s_k<p^{r-1}$ it follows by the induction hypothesis and Lemma \ref{L1}. The case $j_k\geq p(s_k-1)$ and $j_i\geq ps_i$ is analogous. 

Again by Lemma \ref{Lie action on I(V)}, we obtain
\begin{eqnarray}
\delta_{(k,x^I)}(v)&=&\frac{\partial}{\partial x_k}x^Iv+\sum_{j>k}\frac{\partial}{\partial x_j} x^I
E_{jk}v\label{Fc2}
\end{eqnarray}
as $E_{kk}v=m_kv=v$.
 
For $j>k$ we get
\[x_k\frac{\partial}{\partial x_i}\frac{\partial}{\partial x_j}x^IE_{jk}v\in G(n,r)v\]
by (\ref{IHc2}) and Lemma \ref{L2}(1).

Now we apply the operator $\delta_{(i,x_i^2)}\circ\delta_{(i,x_k)}$ to $\delta_{(k,x^I)}(v)$. For $j>k$, we get
\begin{eqnarray*}
&&\delta_{(i,x_i^2)}\left(\delta_{(i,x_k)}\left(\frac{\partial}{\partial x_j} x^I
E_{jk}v\right)\right)\\
&=&\underbrace{\delta_{(i,x_i^2)}\left(x_k\frac{\partial}{\partial x_i}\frac{\partial}
{\partial x_j}x^IE_{jk}v\right)}_{\in G(n,r)v}+\delta_{(i,x_i^2)}\left(\frac{\partial}{\partial x_j} x^IE_{ki}E_{jk}v\right)\\
\end{eqnarray*}
But using $E_{ii}E_{ki}E_{jk}v=m_iE_{ki}E_{jk}v$ Lemma \ref{L3}(4), we get for $j>k$
\begin{eqnarray*}
\delta_{(i,x_i^2)}\left(\frac{\partial}{\partial x_j} x^IE_{ki}E_{jk}v\right)&=&x_i^2\frac{\partial}{\partial x_i}\frac{\partial}{\partial x_j} x^IE_{ki}E_{jk}v+2x_i\frac{\partial}{\partial x_j} x^IE_{ii}E_{ki}E_{jk}v\\
&=&x_i^2\frac{\partial}{\partial x_i}\frac{\partial}{\partial x_j} x^IE_{ki}E_{jk}v+2m_ix_i\frac{\partial}{\partial x_j} x^IE_{ki}E_{jk}v\\
&\in&G(n,r)v
\end{eqnarray*}
by (\ref{IHc2}) and Lemma \ref{L2}(2).

As $\delta_{(i,x_i^2)}(\delta_{(i,x_k)}(\delta_{(k,x^I)}(v)))\in G(n,r)v$, we obtain
\[\delta_{(i,x_i^2)}\left(\delta_{(i,x_k)}\left(\frac{\partial}{\partial x_k}x^Iv\right)\right)\in G(n,r)v\]
by (\ref{Fc2}). That is, by using 
\[x_k\frac{\partial}{\partial x_i}\frac{\partial}{\partial x_k}x^I=(ps_k-1)\frac{\partial}{\partial x_i}x^I=-\frac{\partial}{\partial x_i}x^I\]
we get
\[\delta_{(i,x_i^2)}\left(-\frac{\partial}{\partial x_i}x^Iv+
\frac{\partial}{\partial x_k}x^IE_{ki}v\right)\in G(n,r)v\]
Using $E_{ii}E_{ki}v=m_iE_{ki}v$ Lemma \ref{L3}(4), we obtain
\begin{eqnarray*}
\delta_{(i,x_i^2)}\left(\frac{\partial}{\partial x_k}x^IE_{ki}v\right)&=&x_i^2\frac{\partial}{\partial x_i}\frac{\partial}{\partial x_k}x^IE_{ki}v+
2x_i\frac{\partial}{\partial x_k}x^IE_{ii}E_{ki}v\\
&=&x_i^2\frac{\partial}{\partial x_i}\frac{\partial}{\partial x_k}x^IE_{ki}v+
2m_ix_i\frac{\partial}{\partial x_k}x^IE_{ki}v\\
&\in&G(n,r)v
\end{eqnarray*}
by (\ref{IHc2}) and Lemma \ref{L2}(1). Hence
\[\delta_{(i,x_i^2)}\left(\frac{\partial}{\partial x_i}x^Iv\right)\in G(n,r)v\]

Using $E_{ii}v=(m_i+1)v$ Lemma \ref{L3}(3), we obtain
\begin{eqnarray*}
\delta_{(i,x_i^2)}\left(\frac{\partial}{\partial x_i}x^Iv\right)&=&x_i^2\frac{\partial}{\partial x_i}\frac{\partial}{\partial x_i}x^Iv+
2x_i\frac{\partial}{\partial x_i}x^IE_{ii}v\\
&=&2x^Iv-
2(m_i+1)x^Iv\\
&=&-2m_ix^Iv\\
&\in&G(n,r)v
\end{eqnarray*}
That is, we get
\[2m_ix^Iv\in G(n,r)v\]
which is nonzero as $\Char(k)=p\neq 2$ and $1\leq m_i\leq p-1$. Hence
\[x_1^{ps_1-1}\cdots x_n^{ps_n-1}v=x^Iv\in G(n,r)v\]
which finishes the proof for $m_k=1$.
\end{Case}
Finally, we proved the Proposition. 
\end{proof}

\section{The Grothendieck Ring}
For a $k$-group scheme $G$ denote by
\[\Rep(G):=K_0(G\mathrm{-rep})\]
the Grothendieck ring of finite dimensional representations. 
This is a free abelian group where a $\Z$-basis is given by the classes of the irreducible representations.

Using the description of irreducible $G(n,r)$-representations of the previous section, we are now ready to describe $\Rep(G(n,r))$.
Recall that the functor $\I$ is exact. 
\begin{Thm}
Assume that $\Char(k)\neq 2$. Then the maps
\[\binom{\I}{P_1^{\ast}}:\Rep(\GL_n)\oplus\Rep((\GL_n)^{(1)})\rightarrow\Rep(G(n,1))\]
and for $r\geq 2$
\[\binom{\I}{T_r^{\ast}}:\Rep(\GL_n)\oplus\Rep(G(n,r-1)^{(1)})\rightarrow\Rep(G(n,r))\]
are surjective morphisms of abelian groups.
\end{Thm}
\begin{proof}
Recall that a $\Z$-basis of $\Rep(G(n,r))$ is given by $[L(\lambda,G(n,r))]$ with $\lambda\in X(T)_+$. 
That is, it suffices to show that these classes lie in the respective image.

We start with the case $r=1$ and the map $\binom{\I}{P_1^{\ast}}$. For $r(\lambda)=0$, we know by Proposition 
\ref{soc for rlambda=0} that
\[L(\lambda,G(n,1))=P_1^{\ast}L(s(\lambda))\]
which shows the claim. For $r(\lambda)=\varepsilon_1+\ldots+\varepsilon_i$, we know by Proposition 
\ref{soc for rlambda=fw} that
\[L(\lambda,G(n,1))\cong \im(d_i\otimes\id:\Omega_1^{i-1}\otimes_k L(s(\lambda))^{[1]}\rightarrow \Omega_1^i\otimes_k L(s(\lambda))^{[1]})\]
the images in the twisted de Rham complex. We also know that its cohomology computes as
\[H^i(\Omega_1^{\bullet}\otimes_k L(s(\lambda))^{[1]})\cong P_1^{\ast}(\Lambda^iU^{(1)}\otimes_k L(s(\lambda)))\]
That is, the cohomology classes lie in the image of $\binom{\I}{P_1^{\ast}}$. As 
\[\Omega_1^i\otimes_k L(s(\lambda))^{[1]})=\I(\Lambda^i U\otimes_k L(s(\lambda))^{[1]})\]
also the classes of the objects of the complex lie in this image. As
\[[\im(d_n\otimes\id)]=[\Omega_1^n\otimes_k L(s(\lambda))^{[1]})]+[H^n(\Omega_1^{\bullet}\otimes_k L(s(\lambda))^{[1]})]\]
we get the claim for $r(\lambda)=\varepsilon_1+\ldots+\varepsilon_n$. Now let $i<n$. Then
\begin{eqnarray*}
&&[\im(d_i\otimes\id)]\\\
&=&[\Omega_1^i\otimes_k L(s(\lambda))^{[1]})]-[\im(d_{i+1}\otimes\id)]-[H^i(\Omega_1^{\bullet}\otimes_k L(s(\lambda))^{[1]})]
\end{eqnarray*}
inductively provides the claim for $r(\lambda)=\varepsilon_1+\ldots+\varepsilon_i$. Finally, in the case that $r(\lambda)$ 
is neither $0$ nor a fundamental weight, Proposition \ref{soc for rlambda not 0 or fw} provides
\[L(\lambda,G(n,r))=\I (L(\lambda))\]
which finishes the case $r=1$. 

Now let $r\geq 2$. That is, we consider the map $\binom{\I}{T_r^{\ast}}$. Then for $r(\lambda)=0$, we get
\[L(\lambda,G(n,r))=T_r^{\ast}L(s(\lambda),G(n,r-1)^{(1)})\]
by Proposition \ref{soc for rlambda=0} which shows the claim. For the case $r(\lambda)=\varepsilon_1+\ldots+\varepsilon_i$ 
we use the twisted de Rham complex
\[\Omega_r^i\otimes_k L(s(\lambda))^{[1]})=\I(\Lambda^i U\otimes_k L(s(\lambda))^{[1]})\]
As the image of the $i$-th differential is $L(\lambda,G(n,r))$ by Proposition \ref{soc for rlambda=fw} and its 
cohomology computes as 
\[H^i(\Omega_r^{\bullet}\otimes_k L(s(\lambda))^{[1]})\cong T_r^{\ast}((\Omega_{r-1}^i)^{(1)}\otimes_k L(s(\lambda)))\]
the claim follows in the same fashion as in the case $r=1$. Again the case where $r(\lambda)$ is neither $0$ nor fundamental follows from
Proposition \ref{soc for rlambda not 0 or fw}.
\end{proof}
\begin{Rem}
Let $\Char(k)=2$. Then for $n=r=1$, the case where $r(\lambda)$ is neither $0$ nor fundamental does not occur. But this is the only case
where the assumption $\Char(k)\neq 2$ is necessary in order to apply Proposition \ref{soc for rlambda not 0 or fw}. That is, for 
$n=r=1$, we also get a surjection
\[\binom{\I}{P_1^{\ast}}:\Rep(\mathbb{G}_m)\oplus\Rep((\mathbb{G}_m)^{(1)})\rightarrow\Rep(G(1,1))\]
In the case that $r\geq 2$ or $n\geq 2$, the author does not know wether the maps in question are 
surjective.
\end{Rem}

So far, we only described the abelian group structure of $\Rep(G(n,r))$. We also want to understand its ring structure and 
the kernels of the maps of the theorem. For both topics, the main tool is the restriction map
\[\res:\Rep(G(n,r))\rightarrow \Rep(\GL_n)\]
In fact, this map is injective.
\begin{Lem}
The restriction map
\[\res:\Rep(G(n,r))\rightarrow\Rep(\GL_n)\]
is injective.
\end{Lem}
\begin{proof}
As $\mathbb{G}_m\subset \GL_n\subset G(n,r)$, we can study $\mathbb{G}_m$-weight spaces for 
$G(n,r)$-representations. Furthermore, the $\mathbb{G}_m$-weight space filtration of 
each $GL_n$-representation is $\GL_n$-invariant. Denote by $\deg:X(T)\rightarrow \Z$ 
the degree map which is induced by $\deg(\epsilon_i)=1$ for all $i=1,\ldots,n$. Then for all $n\in\Z$ and 
$\GL_n$-representations $V$, we get for the $n$-th $\mathbb{G}_m$-weight space
\[V_n=\bigoplus_{\lambda\in\deg^{-1}(n)}V_\lambda\]
Now consider $L(\lambda,G(n,r))$ for $\lambda\in X(T)_+$, whose classes form a $\Z$-basis of $\Rep(G(n,r))$. 
Its lowest $\mathbb{G}_m$-weight space is $L(\lambda)$ of weight $s=\deg(\lambda)$. Hence, we obtain
\[[\res L(\lambda,G(n,r))]=[L(\lambda)]+\sum_{\substack{\mu\in X(T)_+\\\deg(\mu)>s}}m_\mu[L(\mu)]\in \Rep(\GL_n)\]
where $m_{\mu}$ is the multiplicity of $L(\mu)$ in $L(\lambda,G(n,r))$. As $[L(\lambda)]$ with $\lambda\in X(T)_+$ 
form a $\Z$-basis of $\Rep(\GL_n)$, $\res$ maps a $\Z$-basis of $\Rep(G(n,r))$ to a linearly independent set. 
Hence $\res$ is injective.
\end{proof}
Now $\Rep(\GL_n)$ is a
$\Rep((\GL_n)^{(1)})$-algebra by the pullback of the first Frobenius
\[F^{\ast}:\Rep((\GL_n)^{(1)})\rightarrow\Rep(\GL_n)\]
Note that under the identification $\GL_n^{(1)}\cong\GL_n$ the ring
homomorphism $F^{\ast}$ is the $p$-th Adams operation $\psi^p$ on the
$\lambda$-ring $\Rep(\GL_n)$. We fix an $r$ and endow the direct
sum
\[\Rep(\GL_n)\oplus \Rep((\GL_n)^{(1)})\]
with a ring structure by assigning
\[(b,a)\cdot(b',a'):=(F^{\ast}(a)b'+F^{\ast}(a')b+[R(n,r)]bb',aa')\]
where $[R(n,r)]$ is the class of the $G(n,r)$-representation $R(n,r)$
restricted to the subgroup $\GL_n$. Now the inclusion
\[\Rep((\GL_n)^{(1)})\hookrightarrow \Rep(\GL_n)\oplus \Rep((\GL_n)^{(1)})\]
gives an $\Rep((\GL_n)^{(1)})$-algebra structure. Furthermore the map
\[\binom{[R(n,r)]\cdot(-)}{F^{\ast}}:\Rep(\GL_n)\oplus
\Rep((\GL_n)^{(1)})\rightarrow\Rep(\GL_n)\]
is an $\Rep((\GL_n)^{(1)})$-algebra morphism. Finally note that
\[\res\circ \I=[R(n,r)]\cdot(-)\]
on $\Rep(\GL_n)$.

Let us consider the case $r=1$. In order to show that
\[\binom{\I}{P_1^{\ast}}:\Rep(\GL_n)\oplus \Rep((\GL_n)^{(1)})\rightarrow
\Rep(G(n,1))\]
is a ring homomorphism, it suffices to show this after composition with $\res$.
But this is precisely the ring homomorphism $\binom{[R(n,1)]\cdot(-)}{F^{\ast}}$
considered above.

Now consider the case $r\geq 2$. Then we introduce a ring structure on
\[\Rep(\GL_n)\oplus\Rep(G(n,r-1)^{(1)})\]
by
\[(b,a)\cdot(b',a'):=(F^{\ast}(\res(a))b'+F^{\ast}(\res(a'))b+[R(n,r)]bb',aa')\]
where we use the restriction
\[\res:\Rep(G(n,r-1)^{(1)})\hookrightarrow \Rep((\GL_n)^{(1)})\]
This makes
\[\Rep(\GL_n)\oplus\Rep(G(n,r-1)^{(1)})\xrightarrow{\id\oplus\res}
\Rep(\GL_n)\oplus \Rep((\GL_n)^{(1)})\]
into a ring injection. In order to show that
\[\binom{\I}{T_r^{\ast}}:\Rep(\GL_n)\oplus
\Rep(G(n,r-1)^{(1)})\rightarrow
\Rep(G(n,r))\]
is a ring homomorphism, we compose with $\res$ again. But this composition
coincides with the ring injection $\id\oplus\res$ followed by the ring
homomorphism
\[\binom{[R(n,r)]\cdot(-)}{F^{\ast}}:\Rep(\GL_n)\oplus
\Rep((\GL_n)^{(1)})\rightarrow\Rep(\GL_n)\]
considered above.

Now we come to kernel elements. For all $r\geq 0$, let us consider the element
\[\delta_r:=\sum_{i=0}^n (-1)^{n-i}[\Lambda^i U^{(r)}]\in\Rep((\GL_n)^{(r)})\]
where $U=k^n$ and for $r\geq 1$ the $r-1$-th Frobenius
\[F^{r-1}:(\GL_n)^{(1)}\rightarrow(\GL_n)^{(r)}\]
\begin{Prop}\label{Prop}
For all $r\geq 1$, the kernel of
\[\binom{[R(n,r)]\cdot(-)}{F^{\ast}}:\Rep(\GL_n)\oplus
\Rep((\GL_n)^{(1)})\rightarrow\Rep(\GL_n)\]
is generated by $(\delta_0,-(F^{r-1})^{\ast}(\delta_r))$ as an
$\Rep((\GL_n)^{(1)})$-module.
\end{Prop}
For $r=1$, $\Rep(G(n,1))$ is a $\Rep((\GL_n)^{(1)}$-algebra by $P_1^{\ast}$. Then
\[\binom{\I}{P_1^{\ast}}:\Rep(\GL(U))\oplus
\Rep(\GL(U)^{(1)})\rightarrow
\Rep(G(n,1))\]
is an $\Rep((\GL_n)^{(1)}$-algebra map. The Proposition implies that its kernel
is generated by $(\delta_0,-\delta_1)$ as an $\Rep((\GL_n)^{(1)})$-module as
\[\res\circ\binom{\I}{P_1^{\ast}}=\binom{[R(n,1)]\cdot(-)}{F^{\ast}}\]
and $\res:\Rep(G(n,1))\rightarrow\Rep(\GL_n)$ is an injective $\Rep((\GL_n)^{(1)})$-algebra map. 

For $r\geq 2$, we do not even have a $\Rep((\GL_n)^{(1)}$-module structure on $\Rep(G(n,r))$. But we can study the commutative diagram
\[\xymatrix{\Rep(\GL_n)\oplus\Rep(G(n,r-1)^{(1)}\ar[d]_{\id\oplus\res}\ar[rr]^-{\binom{\I}{T_r^{\ast}}}&&\Rep(G(n,r))\ar[d]^{\res}\\
\Rep(\GL_n)\oplus\Rep((\GL_n)^{(1)})\ar[rr]^-{\binom{[R(n,r)]\cdot(-)}{F^{\ast}}}&&\Rep(\GL_n)}\]
and get the following.
\begin{Cor}\label{Cor}
 For all $r\geq 2$, the image 
\[(\id\oplus\res)(\Ker(\I+T_r^{\ast}))\subset \Rep(\GL_n)\oplus\Rep((\GL_n)^{(1)})\]
coincides with the kernel of
\[\binom{[R(n,r)]\cdot(-)}{F^{\ast}}:\Rep(\GL_n)\oplus\Rep((\GL_n)^{(1)})\rightarrow\Rep(\GL_n)\]
Furthermore
\[\Ker\binom{\I}{T_r^{\ast}}=\{(\delta_0F^{\ast}(a),-\I_{r-1}(\delta_1 a))\mid a\in\Rep((\GL_n)^{(1)})\}\]
\end{Cor}
Before we prove the Proposition and the Corollary, we need another tool, the \emph{character map}
\begin{eqnarray*}
\Rep(\GL_n) &\xrightarrow{\ch}&\Z[X(T)]\\
 {[}V{]}&\mapsto&\sum_{\lambda\in X(T)}\dim(V_{\lambda})e(\lambda)
\end{eqnarray*}
where $e(\lambda)$ is the basis element of $\Z[X(T)]$ corresponding to $\lambda\in X(T)$. Due to the 
highest weight characterization of the $L(\lambda)$, the character map maps this $\Z$-basis of $\Rep(\GL_n)$ to a 
linearly independent set in $\Z[X(T)]$. Hence it is injective. It is well known that its image is precisely $\Z[X(T)]^W$ 
where $W=S_n$ is the Weyl group.

Now let us write the indeterminant $t_i$ for $e(\varepsilon_i)$ and denote by $s_i$ the $i$-th elementary symmetric 
polynomial in the $t_i$. Then
\[\Z[X(T)]=\Z[t_1^{\pm 1},\ldots,t_n^{\pm 1}]\]
and
\[\Z[X(T)]^W=\Z[s_1,\ldots,s_n,s_n^{-1}]\]
As $F^{\ast}:\Rep((\GL_n)^{(1)})\rightarrow \Rep(\GL_n)$ acts on the $T$-weights by the $p$-th power, it 
corresponds to the $p$-th Adams operation $\psi^p$ on $\Z[X(T)]$ which is given by $\psi^p(t_i)=t_i^p$. Note that
\begin{eqnarray*}
\ch([\Lambda^i U])&=&s_i\\
\ch(\delta_r)&=&\sum_{i=0}^n(-1)^{n-i}s_i=\prod_{i=1}^n(t_i-1)=:\delta \\
\ch([R(n,r)])&=&\prod_{i=1}^n\frac{t_i^{p^r}-1}{t_i-1}=:U_r
\end{eqnarray*}
Hence
\[\ch((F^r)^{\ast}(\delta_r))=(\psi^p)^r(\delta)=U_r\delta=\ch([R(n,r]\delta_0)\]
which implies that $(\delta_0,-(F^{r-1})^{\ast}(\delta_r))$ lies in the kernel of
\[\binom{[R(n,r)]\cdot(-)}{F^{\ast}}:\Rep(\GL_n)\oplus
\Rep((\GL_n)^{(1)})\rightarrow\Rep(\GL_n)\]
Before we prove the whole Proposition, we first deduce the Corollary.
\begin{proof}[Proof of \ref{Cor}] 
According to the Proposition, the kernel of $\binom{[R(n,r)]\cdot(-)}{F^{\ast}}$ is generated by 
$(\delta_0,-(F^{r-1})^{\ast}(\delta_r))$ as an $\Rep((\GL_n)^{(1)})$-module. Hence the kernel elements are 
those of the form
\[(\delta_0F^{\ast}(a),-(F^{r-1})^{\ast}(\delta_r)a)\]
for all $a\in\Rep((\GL_n)^{(1)})$. Furthermore, $(\psi^p)^{r-1}(\delta)=U_{r-1}\delta$. Hence
\[(F^{r-1})^{\ast}(\delta_r)a=[R(n,r-1)^{(1)}]\delta_1 a=\res\I_{r-1}(\delta_1 a)\]
lies in the image of
\[\res:\Rep(G(n,r-1)^{(1)}\rightarrow \Rep((\GL_n)^{(1)})\]
That is, 
\[\Ker\binom{[R(n,r)]\cdot(-)}{F^{\ast}}\subset \im(\id\oplus\res)\]
Now the assertion follows from the commutativity of
\[\xymatrix{\Rep(\GL_n)\oplus\Rep(G(n,r-1)^{(1)}\ar[d]_{\id\oplus\res}\ar[rr]^-{\binom{\I}{T_r^{\ast}}}&&\Rep(G(n,r))\ar[d]^{\res}\\
\Rep(\GL_n)\oplus\Rep((\GL_n)^{(1)})\ar[rr]^-{\binom{[R(n,r)]\cdot(-)}{F^{\ast}}}&&\Rep(\GL_n)}\]
and the injectivity of $\res$.
\end{proof}
Now we give the proof of the Proposition.
\begin{proof}[Proof of \ref{Prop}]
We will prove the Proposition in terms of $\Z[X(T)]$. In fact, we prove the following claim: The kernel of the map
\[\binom{U_r\cdot(-)}{\psi^p}:\Z[X(T)]\oplus\Z[X(T)]\rightarrow\Z[X(T)]\]
coincides with
\[\{(\delta\psi^p(a),-(\psi^p)^{r-1}(\delta)a)\mid a\in\Z[X(T)]\}\]
Then the Proposition follows by passing to $S_n$-invariants as $\Z[X(T)]$ is factorial.

We start with the case $r=1$. As $\psi^p$ and $U_1\cdot(-)$ are injective, it suffices to show
\[\Z[t_1^{\pm p},\ldots,t_n^{\pm p}]\cap U_1\Z[t_1^{\pm 1},\ldots,t_n^{\pm 1}]=(U_1\delta)\Z[t_1^{\pm p},\ldots,t_n^{\pm p}]\]
For this, define the ideals
\[P_i:=\langle \frac{t_i^p-1}{t_i-1}\rangle\subset\Z[t_1^{\pm 1},\ldots,t_n^{\pm 1}]\]
and 
\[Q_i:=\left\langle t_i^p-1\right\rangle\Z[t_1^{\pm p},\ldots,t_n^{\pm p}]\]
Both $P_i$ and $Q_i$ are prime ideals of height $1$ as they are generated by irreducible elements in factorial rings.
We claim that the inclusion
\[Q_i\subset P_i\cap \Z[t_1^{\pm p},\ldots,t_n^{\pm p}]\]
is an equality. As $\Z[t_1^{\pm p},\ldots,t_n^{\pm p}]$ is factorial, it is integrally closed. Thus the ``going-down'' 
Theorem \cite{Mat86}*{Theorem 9.4} implies that $P_i\cap \Z[t_1^{\pm p},\ldots,t_n^{\pm p}]$
is also a prime ideal of height $1$. Whence the equality. As
\[\langle U_1\rangle=P_1\cdots P_n=P_1\cap\ldots\cap P_n\subset \Z[t_1^{\pm 1},\ldots,t_n^{\pm 1}]\]
and
\[\langle U_1\delta\rangle=Q_1\cdots Q_n=Q_1\cap\ldots\cap Q_n\subset \Z[t_1^{\pm p},\ldots,t_n^{\pm p}]\]
the case $r=1$ follows.

Now we consider the case $r\geq 2$. Then we can factor our map as follows since $U_1\psi^p(U_{r-1})=U_r$.
\[\xymatrix{\Z[X(T)]\oplus\Z[X(T)]\ar[rr]^-{\binom{U_r\cdot(-)}{\psi^p}}\ar[d]_-{\psi^p(U_{r-1})\cdot(-)\oplus\id}&&\Z[X(T)]\\
 \Z[X(T)]\oplus\Z[X(T)]\ar[urr]_-{\binom{U_1\cdot(-)}{\psi^p}}}\]
By the case $r=1$ the kernel of the map $U_1\cdot(-)+\psi^p$ consists of the elements of the form
\[(\delta\psi^p(a),-\delta a)\]
for $a\in\Z[X(T)]$. As $\psi^p(U_{r-1})\cdot(-)\oplus\id$ is injective and no prime factor $t_i-1$ of $\delta$ divides 
$\psi^p(U_{r-1})$, the images of the elements of the kernel of $\binom{U_r\cdot(-)}{\psi^p}$ are those of the above type where $U_{r-1}$ divides $a$. 
As $\delta U_{r-1}=(\psi^p)^{r-1}(\delta)$, the claim for $r\geq 2$ follows.
\end{proof}

\end{document}